\documentclass{article}
\usepackage[T2A]{fontenc}
\usepackage[cp1251]{inputenc}
\usepackage[russian,english]{babel}
\usepackage[tbtags]{amsmath}
\usepackage{amsfonts,amssymb,mathrsfs,amscd}
 \usepackage{msb-a}
\JournalName{}
\numberwithin{equation}{section}
\theoremstyle{plain}

\newtheorem{theorem}{Theorem}
\theoremstyle{plain}

\newtheorem{proposition}{Proposition}
\theoremstyle{definition}
\newtheorem{proof}{Proof}

\newtheorem{remark}{Remark}
\begin{document}
\title{Large gaps between sums of two squares}
\author{A.\,B.~Kalmynin}
\address{National Research University Higher School of Economics, Russian Federation, 6 Usacheva Str., Moscow, Russia, 119048}
\email{alkalb1995cd@mail.ru}

\author{S.\,V.~Konyagin}
\address{Steklov Mathematical Institute, 8 Gubkin Str., Moscow, Russia, 119991}
\email{konyagin@mi-ras.ru}

\date{}
\udk{}
\maketitle
\begin{fulltext}
\begin{abstract}
\begin{bf}{Abstract:}\end{bf}

Let $\mathcal S=\{s_1<s_2<s_3<\ldots\}$ be the sequence of all natural numbers which can be 
represented as a sum of two squares of integers. For $X\ge2$ we denote by $g(X)$ 
the largest gap between consecutive elements of $\mathcal S$ that do not exceed $X$.
We prove that for $X \to +\infty$ the lower bound
$$g(X)\geq \left(\frac{390}{449}-o(1)\right)\ln X$$
holds.

This estimate is twice the recent estimate by R. Dietmann and C. Elsholtz.
\end{abstract}
\footnotetext{The first author is partially supported by Laboratory of Mirror Symmetry NRU HSE, RF Government grant, ag. \textnumero 14.641.31.0001, the "Russian Young Mathematics" contest, the Simons foundation and the foundation for Advancement of Theoretical Physics and Mathematics "BASIS”.}
\section{Introduction}

Let $\mathcal S=\{s_1<s_2<s_3<\ldots\}$ be the sequence of all natural numbers which can be represented as a sum of two squares of integers. Various questions regarding behavior of the quantity $g(X)=\max\limits_{n:s_{n+1}\leq X}(s_{n+1}-s_n)$, i.e. the largest gap between consecutive elements of $\mathcal S$ that do not exceed $X$, for $X \to +\infty$ were studied by many authors in different contexts. The best lower bounds for $g(X)$ were obtained by P.\,Tur\'an, P.\,Erd\H{o}s \cite{Erd}, I.\,Richards \cite{Rich} and by R.\,Dietmann and C.\,Elsholtz \cite{DiEl}. Proof of the best known upper bound for the function $g(X)$ is contained in paper \cite{BaCh}.

The goal of this article is to obtain a new lower bound for the quantity $g(X)$. Using recent result by R.\,Dietmann and C.\,Elsholtz \cite{DiEl}, we will construct large intervals that contain no elements of $\mathcal S$ and will prove the following statement

\begin{theorem}
For $X \to +\infty$ the lower bound

\begin{equation}
\label{desired}
g(X)\geq \left(\frac{390}{449}-o(1)\right)\ln X=(0.86859\ldots-o(1))\ln X.
\end{equation}

holds.
\end{theorem}

The estimate (\ref{desired}) is twice the estimate of the article \cite{DiEl}.

\section{Main result}

All currently known theorems on large gaps between sums of two squares rely on the following general proposition:

\begin{proposition}
Assume that for every large enough real number $Y$ there exist natural number $P(Y)$ without prime factors of the form $4k+1$ and satisfying the inequality $Y\leq P(Y)\leq \Phi(Y)/2$ for some continuous increasing function $\Phi(Y)$ and a natural number $a(Y)$ not exceeding $P(Y)$ and satisfying for any natural $j\leq Y$ at least one of conditions

\begin{enumerate}
\item There exist an odd prime $p$ and an odd positive integer $k$ with $p^{k+1} \mid P(Y)$ and $a(Y)+j\equiv p^k a_j \pmod {P(Y)}$ for some $a_j$ that is not divisible by $p$.
\item There exist positive integers $k$ and $m$ with $2^{k+2} \mid P(Y)$ and $a(Y)+j \equiv 2^k(4m-1) \pmod {P(Y)}$.
\end{enumerate}
Then for all large enough $X$ the inequality $g(X)\geq \Phi^{-1}(X)$ is true, where $\Phi^{-1}$ is the compositional inverse of the function $\Phi$.
\end{proposition}
\begin{proof}
The interval $I=[1+a(Y);Y+a(Y)]$ lies inside the interval $[1,\Phi(Y)]$ for all $Y$ as $Y+a(Y)\leq 2P(Y)\leq \Phi(Y)$. Let us show that this interval does not contain any elements of the set $\mathcal S$. Indeed, if $n\in I$ then for some $j\leq Y$ one has $n=j+a(Y)$. Now, if the first condition of Proposition 1 holds for $j$ then $n$ is not a sum of two squares, because for some prime number $p$ with $p\equiv 3 \pmod 4$ there is odd number $k$, natural $a_j$ which is not divisible by $p$ and natural $b_j$ such that $P(Y)=p^{k+1}b_j$ and $n=a(Y)+j\equiv p^ka_j \pmod {P(X)}$. From this we obtain

\[
n=p^ka_j+cP(Y)=p^k(a_j+b_jcp).
\]

Therefore, some prime $p$ congruent to 3 modulo 4 has an odd exponent in prime factorization of $n$ and thus $n$ is not in $\mathcal S$.

On the other hand, if for $j$ the second condition holds then analogous argument shows that $n$ is equal to $2^k(4u-1)$ for some positive integer $u$, but the numbers of this form are not sums of two squares.

Hence, for all large enough $Y$ the interval $[1,\Phi(Y)]$ contains a subinterval of length $Y$ that does not intersect with $\mathcal S$. Consequently, 

\[
g(\Phi(Y))\geq Y.
\]

Choosing $Y=\Phi^{-1}(X)$ we get the desired result.
\end{proof}

I.\,Richards chose the following number to be $P(Y)$ in his construction: 

\[
P_1(Y)=\prod_{\substack{p\leq 4Y \\ p\equiv 3 \pmod 4}}p^{\beta_p+1},
\]

where $\beta_p=[\ln(4Y)/\ln p]$. In this case $a(Y)$ is the solution of congruence $4a(Y)\equiv -1 \pmod {P_1(Y)}$. Dietmann and Elsholtz the product of the form

\[
2^kP_2(Y,k)=2^k\prod_{p\in A_k(Y)}p^{\beta_p+1}
\]

as their $P(Y)$. Here $k$ is a large positive integer and $A_k(Y)$ is some subset of prime numbers $p\leq 4Y$ that are congruent to 3 modulo 4. In this construction, the number $a(Y)$ is the solution of congruences $4a(Y)\equiv -1\pmod {P_2(X,k)}$ and $a(Y) \equiv 0 \pmod {2^k}$.

In the fist case $\ln P(Y)=(4+o(1)) Y$ so that $\exp((4+\varepsilon)Y)$ is an admissible choice of $\Phi(Y)$ for arbitrarily small $\varepsilon>0$. In the second case we have $\ln P(Y)=(449/195+\varepsilon_k+o(1))Y$, where $\varepsilon_k\to 0$ as $k\to \infty$, which leads to the inequality $g(X)\geq (195/449-o(1))\ln X$. In what follows, we are going to show that under certain additional restrictions on $P(Y)$ and $\Phi(Y)$ the conditions of Propostion 1 imply the lower bound $g(X)\geq \Phi^{-1}(X^{2-o(1)})$.

\begin{theorem}
Assume that for $Y$ large enough there are $P(Y)$ and $a(Y)$ satisfying conditions of Proposition 1. Suppose that $\ln \Phi(Y)=o(Y\ln Y)$ and that small prime factors make a small contribution to the size of $P(Y)$. Namely, we are going to assume that the relation

\[
\lim_{\varepsilon\to 0}\limsup_{Y\to +\infty} \frac{\ln P(Y,\varepsilon)}{\ln P(Y)}=0
\]

is true, where $P(Y,\varepsilon)=\prod\limits_{\substack{p^k \mid\mid P(Y) \\ p\leq \varepsilon Y}}p^k$. Here for integer $N$, nonnegative integer $k$ and prime $p$ expression $p^k \mid\mid N$ means that $p^k \mid N$ but $p^{k+1}\nmid N$. Assume also that for all $\varepsilon>0$ and all large enough  $Y$ the inequality $Y\leq \Phi(Y)^\varepsilon$ holds. Then for any $\varepsilon>0$ and large enough $X$ we have $g(X)\geq \Phi^{-1}(X^{2-\varepsilon})$.
\end{theorem}

\begin{proof}
Let us choose $\delta=\delta(\varepsilon)<1$ so that for large enough $Y$ we have $\ln P(Y,\delta)<\varepsilon \ln P(Y)$. Set

\[
\mathcal P(Y)=P(Y,\delta)\prod_{\substack{p\mid P(Y) \\ p>\delta Y}}p^{\gamma_p},
\]

where $\gamma_p=1$ if there is a natural $j\leq Y$ with $p^{k+1} \mid P(Y)$ and $a(Y)+j \equiv p^k a_j \pmod {P(Y)}$ for some odd $k$ with $a_j$ is coprime to $p$. Let $\gamma_p=0$ otherwise. Here $a(Y)$ is the same as in Proposition 1. 

Note that  $\mathcal P(Y)\leq P(Y)^{1/2+\varepsilon/2}$. Indeed, every exponent $\gamma_p$ is at least two times smaller than the exponent of $p>\delta Y$ in factorization of $P(Y)$, therefore

\[
\mathcal P(Y) \leq P(Y,\delta)\sqrt{\frac{P(Y)}{P(Y,\delta)}}=\sqrt{P(Y,\delta)P(Y)}<P(Y)^{1/2+\varepsilon/2}
\]

by the choice of $\delta$. Choose now the natural number $a_0(Y)$ such that the congruence $a_0(Y) \equiv a(Y) \pmod {\mathcal P(Y)}$ and inequalities $0<a_0(Y)\leq \mathcal P(Y)$ hold. Define the family of intervals $I_n=[1+a_0(Y)+n\mathcal P(Y);Y+a_0(Y)+n\mathcal P(Y)]$, where the variable $n$ takes integer values with $0\leq n\leq \delta Y$. Let us show that at least one of the constructed intervals does not contain any sum of two squares. Indeed, assume that $m\in I_n$ is an element of $\mathcal S$. As $m$ lies in $I_n$, for some $j\leq Y$ the equality $m=a_0(Y)+j+n\mathcal P(Y)$ holds. By the definition of $a(Y)$ and $P(Y)$, at least one of the following conditions holds:

\begin{itemize}
\item There are natural $k$ and $m$ with $2^{k+2} \mid P(Y)$ and $a(Y)+j \equiv 2^k(4m-1) \pmod {P(Y)}$. In this case we also have $a_0(Y)+j\equiv 2^k(4m-1) \pmod {\mathcal P(Y)}$ and $2^{k+2}\mid \mathcal P(Y)$, therefore $m$ cannot be the sum of two squares, which is a contradiction.

\item There are an odd prime $p$ and an odd natural number $k$ with $p^{k+1} \mid P(Y)$ and $a(Y)+j \equiv p^ka_j \pmod {P(Y)}$ for some $a_j$ that is not divisible by $p$. If $p\leq \delta Y$ then these congruences and divisibilities remain true for $a_0(Y)$ and $\mathcal P(Y)$, which once again leads us to contradiction. If, on the contrary, $p>\delta Y$ then necesserily $\gamma_p=1$ and hence $m=a_0(Y)+j+n\mathcal P(Y) \equiv 0 \pmod p$. As $p\equiv 3 \pmod 4$ and $m$ is the sum of two squares, we have $p^2 \mid m$.
\end{itemize}

Notice now, that for fixed $j\leq Y$ and $p>\delta Y$ there exists at most one $n$ such that $a_0(Y)+j+n\mathcal P(Y)$ is divisible by $p^2$. Indeed, otherwise for two distinct $0\leq n_1,n_2\leq \delta Y$ we have $a_0(Y)+j+n_1\mathcal P(Y)\equiv a_0(Y)+j+n_2\mathcal P(Y) \pmod {p^2}$. As $p^2 \nmid \mathcal P(Y)$ we obtain $n_1 \equiv n_2 \pmod p$, which is impossible because $0<|n_1-n_2|\leq \delta Y<p$. Furthermore, for fixed prime $p>\delta Y$ there are at most $1+1/\delta\leq 2/\delta$ numbers $j\leq Y$ with $a_0(Y)+j\equiv 0 \pmod p$, as any two numbers having this property are congruent modulo $p$. Thus, the total amount of  $n$ for which $I_n$ contains a sum of two squares is at most the number of all <<exceptional>> pairs $(j,p)$, i.e. at most $2F/\delta$, where $F$ is the number of prime factors $p>\delta Y$ of $P(Y)$. Clearly, $P(Y)\geq (\delta Y)^F$ so $F\leq \ln P(Y)/\ln(\delta Y)\ll \ln P(Y)/\ln Y\leq \ln \Phi(Y)/\ln Y=o(Y)$ due to conditions of Theorem 2. Therefore, all but $o(Y)$ of intervals $I_n$ do not intersect $\mathcal S$. In particular, for all large enough $Y$ there is at least one interval with this property.

Next, all the resulting intervals lie inside the interval $[1,Y\Phi(Y)^{1/2+\varepsilon/2}]$ because $a_0(Y)+\delta Y\mathcal P(Y)\leq Y+\mathcal P(Y)+\delta Y\mathcal P(Y)\leq YP(Y)^{1/2+\varepsilon/2}\leq Y\Phi(Y)^{1/2+\varepsilon/2}$. By the conditions of our theorem for all large enough $Y$ we have $Y\leq \Phi(Y)^{\varepsilon/2}$. Consequently, for all large $Y$ the inequality $Y\Phi(Y)^{1/2+\varepsilon/2}\leq \Phi^{1/2+\varepsilon}(Y)$ is true, which means that the interval $[1,\Phi(Y)^{1/2+\varepsilon}]$ contains a subinterval of length $Y$ that does not contain sums of two squares. Choosing $Y=\Phi^{-1}(X^{2/(1+2\varepsilon)})$, we obtain the estimate $g(X)\geq \Phi^{-1}(X^{2/(1+2\varepsilon)})$. As $\varepsilon$ was an arbitrary positive real number, this concludes our proof.
\end{proof}

Theorem 1 easily follows, combining results of the paper \cite{DiEl} and Theorem 2:

\begin{proof}[of Theorem 1]

Due to results of \cite{DiEl}, there are $P(Y)$ and $a(Y)$ which satisfy conditions of Proposition 1 and relation $\ln P(Y)=(449/195+o(1))Y$. Furthermore, all the prime factors of $P(Y)$ are at most $4Y$ and all the exponents of $p$ in factorization do not exceed $\beta_p+1=[\ln(4Y)/\ln p]+1\leq 2\beta_p$. Therefore, if $p^\alpha \mid\mid P(Y)$ then $p^\alpha\leq 16Y^2$. It follows that small primes make a small contribution in $P(Y)$. Indeed, if $\varepsilon>0$ then

\[
P(Y,\varepsilon)=\prod_{\substack{p^k\mid\mid P(Y)\\ p\leq \varepsilon Y}} p^k \leq \prod_{\substack{p^k\mid\mid P(Y)\\ p\leq \varepsilon Y}} 16Y^2\leq (16Y)^{2\pi(\varepsilon Y)}.
\]

Consequently,

\[
\ln P(Y,\varepsilon)\leq 2\pi(\varepsilon Y)\ln (4Y)\sim 2\frac{\varepsilon Y}{\ln \varepsilon Y}\ln Y\sim 2\varepsilon Y.
\]

As we also have $\ln P(Y) \gg Y$, we finally get

\[
\lim_{\varepsilon\to 0}\limsup_{Y\to +\infty} \frac{\ln P(Y,\varepsilon)}{\ln P(Y)}\leq \lim_{\varepsilon\to 0} \frac{2\varepsilon}{c}=0,
\]

for some $c>0$. Thus, the construction of Dietmann and Elsholtz satisfies the conditions of Theorem 2. Choosing $\Phi(Y)=\exp((449/195+\varepsilon)Y)$, from Theorem 2 for arbitrary $\varepsilon_1>0$ we get

\[
g(X)\geq \Phi^{-1}(X^{2-\varepsilon_1})=(449/195+\varepsilon)^{-1}(2-\varepsilon_1)\ln X\geq (390/449-2\varepsilon-195/449\varepsilon_1)\ln X.
\]

Small enough values of $\varepsilon$ and $\varepsilon_1$ give us the desired result.
\end{proof}
\begin{remark}

Results of H.\,Iwaniec \cite{Iw} on behavior of Jacobstahl function allow us to show that under assumptions of Theorem 2 the estimate $Y\ll \ln^2 P(Y)/\sqrt{\ln\ln P(Y)}$ holds. In particular, this means that the inequality $Y\ll \Phi(Y)^\varepsilon$ is true for any $\varepsilon>0$ automatically. Further, this implies that current methods cannot prove any estimate stronger than $g(X)\gg \ln^2 X/\sqrt{\ln\ln X}$. On the other hand, some models predict that in fact the order of growth of the quantity $g(X)$ is slightly lower, namely $g(X)\asymp (\ln X)^{3/2}$.
\end{remark}
\end{fulltext}

\end{document}